\documentclass[a4paper,11pt]{amsart}
\usepackage{amsmath,amssymb,amsthm}
\usepackage{mathtools}
\usepackage{todonotes}
\usepackage{bm}
\usepackage{subfigure}
\usepackage[utf8]{inputenc}
\usepackage[T1]{fontenc}
\usepackage{array}
\usepackage[title]{appendix}

\usepackage{hyperref}
\hypersetup{%,filecolor=black% links to local files
colorlinks=true
,breaklinks=true
,urlcolor=magenta
,linkcolor=red% internal document links
,bookmarksopen=false
,citecolor=mediumblue
,anchorcolor=mediumblue%,pagecolor=blue
}%hypersetup
\usepackage{color,svgcolor}
\usepackage[numbers]{natbib} %zhy10 add [numbers]
\newcommand{\N}{\mathbb{N}}
\newtheorem{thm}{Theorem}[section] %zhy4 set counter
\newtheorem{lem}[thm]{Lemma}
\newtheorem{prop}[thm]{Proposition}

\theoremstyle{definition}
\newtheorem{defn}[thm]{Definition}
\theoremstyle{remark}

\newtheorem{example}[thm]{Example}

\newcommand{\ideal}[1]{\langle #1 \rangle}
\newcommand{\QQ}{\mathbb{Q}}
\newcommand{\RR}{\mathbb{R}}
\newcommand{\ZZ}{\mathbb{Z}}
\newcommand{\PP}{\mathbb{P}}
\newcommand{\CC}{\mathbb{C}}

\newcommand{\kk}{k}
\newcommand{\ccc}{c}

\newcommand{\rrr}{r}
\newcommand{\nnn}{n}
\newcommand{\mmm}{m}
\newcommand{\qqq}{Q}
\newcommand{\III}{I}

\newcommand{\PPP}{P}
\newcommand{\GGG}{G}
\newcommand{\HHH}{H}
\newcommand{\XXX}{X}
\newcommand{\YYY}{Y}

\newcommand{\boldp}{\bm{p}}
\newcommand{\boldq}{\bm{q}}

\newcommand{\cg}{c}
\newcommand{\lcm}{\text{\textup{lcm}}}

\newcommand{\groebner}{Gr\"obner }
\newcommand{\lcu}[1]{\textsc{lc}(#1)}
\newcommand{\sing}[1]{\text{\textup{Sing}}(#1)}

\newcommand{\barI}{\bar{I}}
\newcommand{\jac}[1]{\text{Jac}(#1)}
\newcommand{\LM}[1]{\textsc{lm}(#1)}

\newcommand{\Iexp}{I_{\text{exp}}}
\title{The integral closure of a primary ideal is  not always primary}
\author{Nan Li, Zijia Li, Zhi-Hong Yang, and Lihong Zhi}

\begin{document}

\begin{abstract}
In 1936, Krull asked if the integral closure of a primary ideal is
still primary. Fifty years later, Huneke partially answered this question by giving a primary polynomial ideal whose integral closure is not primary in a regular local ring of characteristic $p=2$. We provide counterexamples to Krull's question regarding polynomial rings with any characteristics. We also find that the Jacobian ideal $J$ of the polynomial $f = x^6 + y^6 + x^4 z t + z^3$ given by  Brian{\c{c}}on and Speder in 1975 is a counterexample to Krull's question. Let $V_1$ be the hypersurface defined by $f = 0$ and $V_2$ be its singular locus.
Brian\c{c}on and Speder proved that  Whitney equisingularity does not imply Zariski equisingularity by showing that the pair $(V_1 \setminus V_2,\ V_2)$  satisfies Whitney's conditions around the origin but fails Zariski's equisingular conditions. 
  We discover that the pair $(V_1 \setminus V_2,\ V_2)$ fails Whitney's conditions at the variety of the embedded prime of the integral closure $\bar{J}$, which means that $V_1$ is not Whitney regular along $V_2$. Moreover, we also show that Whitney stratification of this hypersurface is different from the stratification of isosingular sets given by Hauenstein and Wampler, which is related to Thom-Boardman singularity.

\end{abstract}

\maketitle

\setcounter{tocdepth}{1}

\section{Introduction}\label{sec:1}
Krull \cite[p. 577]{krull1936beitrage} asked: {\textit{Ist etwa bei
einem Prim\"arideal $q$ immer auch $q_b$ Prim\"arideal?}}  For monomial ideals, the answer to Krull's question is yes.
The integral closure of a primary monomial ideal is
always primary \cite{salam2002integral}. However, for non-monomial ideals,
Huneke partially answered this question by giving a counterexample
 in the regular local ring
$k[[x,y,z]]$ with char$(k) = 2$ \cite[Example~3.7]{huneke1986primary}. According to \cite{huneke2006integral},
there are no known counterexamples for rings of
characteristic zero.  The integral closure of ideals is related to Whitney equisingularity. For instance,  Teissier \cite{FloresTeissier2018,teissier88limites} gave an algebraic description for Whitney's condition (b)   using the integral closure of the sheaf of ideals, which started the modern equisingularity theory. Gaffney
\cite{gaffney1992integral,gaffney1996multiplicities,gaffney1999specialization} generalized the theory of integral closure of ideals to modules, and made many applications in Whitney equisingularity.

Our main contributions are summarized below:
\begin{itemize}
\item
We answer Krull's question negatively by giving    a sequence of primary ideals
\[\III=\langle x^3, y^3, x^2y, x^2z^n-xy^2 \rangle, n \in \ZZ_+\]
whose integral closures
\[\barI = \ideal{x^3, y^3, x^2y,
x^2 z^\nnn, xy^2 }\]
 are not primary over a field of characteristic zero or positive characteristics.
 Hence,  taking integral
closure of a polynomial primary ideal may create embedded primes. On the other
hand, we also show that there are examples where the given polynomial ideal is not
primary but its integral closure is primary. It implies that taking
integral closure may also remove embedded primes. Therefore,   the
relation between a primary ideal and its integral closure is not clear yet.

\item  We show that the hypersurface defined by $f(x,y,z,t) =  x^6 + y^6 + x^4 z t + z^3=0$   is not Whitney regular along its singular locus. 
This hypersurface was given by Brian{\c{c}}on and Speder \cite{BriSpe1975} to show that  Whitney equisingularity  does not imply  Zariski equisingularity for the set germs at the origin. 
Let  $V_1$ be the hypersurface defined by  $f$, $V_2 = \sing{V_1} = \{(0,0,0,t)\mid t\in \CC\}$ be the singular locus of $V_1$. 
It is known   that the pair $(V_1\setminus V_2,\ V_2)$ satisfies Whitney's conditions (a) and (b) but fails
Zariski equisingularity  at the origin~\cite{BriSpe1975,ParPau2017,wall10gaffney, zariski1977elusive}, which makes the hypersurface $V_1$ as a counterexample for Zariski equisingularity problem \cite{villamayor2000equiresolution,zariski1971some}. 
 We consider Whitney equisingularity for all points of $V_2$ and show that the pair $(V_1\setminus V_2,\ V_2)$ fails Whitney's conditions at $V_3= \{(0,0,0,t)\mid 4t^3+27=0 \}\subset V_2$.%zhy12 was: However, we show that the pair $(V_1\setminus V_2,\ V_2)$ does not satisfy  Whitney's conditions  at all three points in   $V_3= \{(0,0,0,t)\mid 4t^3+27=0 \}$. 
The Jacobian ideal  generated by the partial derivatives of $f$ is primary, but  its integral closure is not, which  gives another  counterexample to Krull's question.  The integral closure of  the Jacobian ideal of $f$ has an embedded prime which happens to be the vanishing ideal of $V_3$.  Furthermore, we also show that Whitney stratification of $V_1$ is different from the stratification given by isosingular sets in \cite{HW2013}.

\end{itemize}

The paper is organized as follows. Section~\ref{sec:basic} is for
basic definitions and properties of integral closures of ideals. In
Section~\ref{sec:examples}, we  show  a sequence of primary ideals  whose  integral closures  are not primary over a field of characteristic zero or positive characteristics. Moreover, we also present an example to show that taking integral
closure may remove embedded primes.  Finally, we compute the integral closure $\bar{J}$  of the Jacobian ideal of $f$ and verify that the pair $(V_1\setminus V_2, V_2)$ fails Whitney's conditions at the variety of the embedded prime of $\bar{J}$.

%Finally, we compute the integral closure of the Jacobian ideal of f and verify that the pair $(V_1\setminus V_2, V_2)$ fails Whitney's conditions at the three points (...).%Finally, we show that   the hypersurface defined by $f(x,y,z,t) =  x^6 + y^6 + x^4 z t + z^3=0$  is not Whitney regular along its singular locus.

% satisfies Whitney’s equisingularity is not correct. There are three points  $(0,0,0,-\sqrt[3]{27/4}\omega)$ ($\omega^3=1$) on the singular locus defined by $x=y=z=0$ such that the Whitney's conditions fail!}

\section{Basic properties}
\label{sec:basic}
Let us first recall some basic definitions from \cite{huneke2006integral}.
\begin{defn}
 Let $I$ be an ideal in a ring $R$. An element $r \in R$ is said to be integral over $I$ if there
exists 
an integer $n$ and elements $a_i \in I^i$, $i = 1,\ldots, n,$
such that
\begin{equation*}
 r^n+a_1r^{n-1}+a_2r^{n-2}+\cdots+a_{n-1}r+a_n=0.
\end{equation*}
\end{defn}
The set of all integral elements over $I$ is called the \emph{integral
closure} of $I$ and is denoted by $\barI$. If $I = \barI$, then $I$
is called integrally closed. Let $J$ be an ideal satisfying $I \subset J$, we say that
$J$ is integral over $I$ if
$\barI \subset J$. 

%zhy2 was $J \subset \bar{I}$.
%In addition, the \emph{monomial ideals} generated by monomials are studied well in the literature \cite{huneke2006integral, salam2002integral}. %ln

\begin{defn}
Let $R$ be the polynomial ring $\kk[x_1,\ldots,x_d]$. For any monomial
$m = x_1^{n_1}x_2^{n_2}\cdots x_d^{n_d}$, its exponent vector is $(n_1,\ldots,n_d) \in \N^d$. For any monomial ideal $I$, the set of all exponent vectors of all the monomials in
$I$ is called the \emph{exponent set} of I.
\end{defn}
The integral closure of a monomial ideal in a polynomial ring is still
a monomial ideal \cite[Proposition~1.4.2]{huneke2006integral}. 
The following proposition is useful for computing the integral closure
of a monomial ideal. 
\begin{prop}\cite[Proposition~1.4.6]{huneke2006integral} 
\label{prop:mo}
The exponent set of the integral closure of a monomial ideal $I$
equals all the integer lattice points in the convex hull of the
exponent set of $I$.
\end{prop}

We give a simple example to show how to 
compute the integral closure of a
monomial ideal using Proposition~\ref{prop:mo}. 
\begin{example}
Let $R=\kk[x,y]$ be a polynomial ring over a field $\kk$. Let $I=\langle
x^2, y^2 \rangle$ be a monomial ideal in $R$.
The exponent set of $I$ is $\Iexp = \{(n_1, n_2) \in \N^2 \mid
n_1 \ge 2, n_2 \ge 2 \}$, see Figure~\ref{fig:exponent}(a). The
convex hull of $\Iexp$ is $\Iexp \cup \{(1,1)\}$, see
Figure~\ref{fig:exponent}(b).
Therefore, by
Proposition~\ref{prop:mo}, 
the integral closure of $I$ is $\barI=\langle x^2,y^2,xy\rangle$.
\begin{figure}[tbhp]
\centering
\subfigure[]{\includegraphics[width=0.3\textwidth]{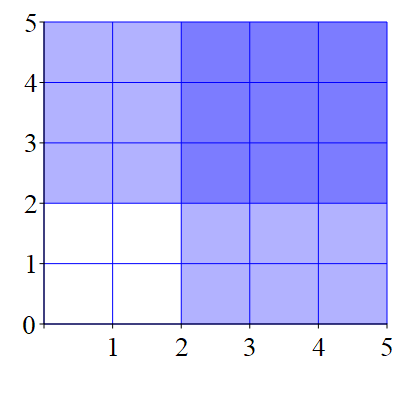}}
\hspace*{2em}
\subfigure[]{\includegraphics[width=0.3\textwidth]{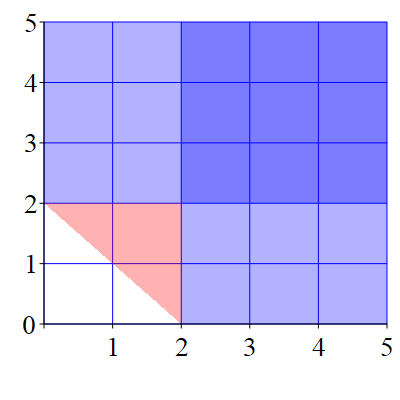}}
\caption{(a) Exponent set of $I$. (b) Exponent set of $\barI$.}
\label{fig:exponent}
\end{figure}
\end{example}

\section{Counterexamples to Krull's question}
\label{sec:examples}

In this section, we answer Krull's question negatively by giving a sequence of primary ideals  whose  integral closures %zhy2 add
 are not primary. The most amazing thing for us to notice is that the  Jacobian ideal of the ideal defined by $f(x,y,z,t) =  x^6 + y^6 + x^4 z t + z^3$
 has already given another counterexample to Krull's question. 
%  
%  In \cite{BriSpe1975},  Brian{\c{c}}on and Speder  mistakenly stated that the hypersurface defined by $f=0$ is Whitney regular along its singular locus. As Zariski showed in \cite{zariski1977elusive}, the blowup along the $t$-axis failed to be equisingular at $t=0$, which led him to reject Whitney equisingularity 
% as a good notion. 
% We show that the hypersurface defined by $f=0$ fails to  be Whitney regular along its singular locus too.       

\subsection{ A set of counterexamples to Krull's question}                                                   %zhy2 add

\begin{example}
\label{ex:counter_ex}

Let $R = k[x,y,z]$ be a polynomial ring over an arbitrary field $k$.
Let 
\begin{equation*}
\III=\langle x^3, y^3, x^2y, x^2z-xy^2 \rangle,
\end{equation*}
where $\III$ is a
primary ideal and its radical ideal is $\PPP=\sqrt{\III}=\langle x, y
\rangle$.
\end{example}

Let $\barI$ be the integral closure of $\III$. 
Using Proposition~\ref{prop:mo}, we obtain
\[\barI = \langle x^3, y^3, x^2y, x^2z, xy^2 \rangle.\]
Its  primary decomposition
is
 \[\barI = \langle x^2, y^3, xy^2 \rangle   
 \cap \langle x^3,y,z \rangle,\]                       
which implies that $\barI$ is not primary. One can also check it via Macaulay2, see Appendix~\ref{app:m1}.

It is interesting to notice that a sequence of counterexamples exist according to the following theorem.

\begin{thm}
\label{cor:ex_for_n}
 Let $n$ be a positive integer, and
 \[\III=\langle x^3, y^3, x^2y, x^2z^n-xy^2 \rangle, ~n \in \ZZ_+.\]
 be a polynomial ideal in $\kk[x,y,z]$.  Then   the ideal $\III$
is primary and its integral closure
\[\barI = \ideal{x^3, y^3, x^2y,
x^2 z^\nnn, xy^2 }\]
 is not primary.

\end{thm}

 Theorem \ref{cor:ex_for_n} follows from  three claims below.   Some  preliminary results
from \cite{BeckerWei1993,shimoyama96} %zhy5 delete , GP2008singular},
 are summarized in
Appendix~\ref{sec:append}.

%%%%%%%%%%%%%%%%%%%%%%%%%%%%%zhy2 add three claims.
\begin{itemize}
\item
{\itshape Claim~1.} The ideal $\III = \ideal{x^3, y^3,
x^2y,x^2z^n-xy^2 } \subset \kk[x,y,z]$ is primary, where $n\in
\ZZ_{+}$.

%zhy4 rewrite Proof of Claim 1
{\itshape Proof of Claim~1.}
%By Defnition~\ref{def:pseudo}, the ideal $\III$ is
%pseudo-primary since $\sqrt{I} = \ideal{x,y} = \PPP$
%is a prime ideal.
Let $\PPP=\sqrt{\III} = \ideal{x,y}$ and $\qqq$ be the primary
component of $\III$ with $\sqrt{\qqq}=\PPP$.
By Lemma~\ref{lem:saturation},
$\qqq = \III : z^\infty$. %zhy5 was $\qqq = \III : \ideal{z}$.
We prove that $\III=\qqq$.

Since $\III\subset\qqq$, we only need to prove that
$\qqq\subset\III$. Suppose that $\III \subsetneq \qqq$, and let
$g\in \qqq$ and $g\notin \III$.
Then $z^\mmm g \in \III$ for some $\mmm\ge 1$.
The set of generators
$ \GGG = \{x^3, y^3,x^2y,x^2z^n-xy^2 \}$ %zhy7
is a Gr\"obner basis of $\III$ with respect to the
lexicographic order $x>y>z$.
Without loss of generality, we assume
that $g$ is reduced with respect to $\GGG$.
%Let $\UUU = \{ z\}$, then $\UUU$ is a maximally
%independent set modulo $\PPP$.
%By Proposition~\ref{prop:pseudo_primary},
%Lemma~\ref{lem:a8}, and Lemma~\ref{lem:ex},
%we have $\qqq = \III : \ideal{z}$.
We may also assume that $\mmm \ge n$, because  %zhy7 add
$z^\mmm g\in \III$ implies that  %zhy7 add
$z^{\mmm+\ell} g \in \III$ for any $\ell \ge 0$.   %zhy7 add
Therefore, we can choose an exponent not less than $n$. %zhy7 add
%zhy5 delete Note that
Let $\LM{G}=\{x^3, y^3, x^2y, x^2z^n \}$ be the set of leading %zhy7
monomials of polynomials in $G$.
Since $\GGG$ is a \groebner basis of $\III$,
$z^\mmm\LM{g}\in \ideal{\LM{G}} =
\ideal{x^3, y^3, x^2y, x^2z^n }$, %zhy7
which implies that either $x^2$ or $y^3$ divides $\LM{g}$.
On the other hand, $\LM{g}$ is not
divisible by any element in $\LM{G}$
because $g$ is reduced with respect to $\GGG$.
Therefore, $\LM{g}=x^2 z^{n-\ell}$ for some $1\le\ell\le n$,
and so
\[g = \cg x^2 z^{n-\ell} + x g_1(y,z) + g_0(y,z)\]
for some %zhy8 was: and so $g = \cg x^2 z^{n-\ell} + x g_1(y,z)$ for some
$\cg \in \kk\setminus \{0\}$ %zhy5 was: $\cg \in \QQ\setminus \{0\}$
and $g_1(y,z), g_0(y,z)\in \QQ[y,z]$. %zhy8 was: and $g_1(y,z)\in \QQ[y,z]$.
%Then the normal form of $z^n g$ with respect to $\GGG$ is
Let $\rrr$ be the remainder on division of $z^\mmm g$ by the
polynomial $x^2 z^n - xy^2 $, %zhy7 $x^2 z - xy^2 $,
\begin{equation*}
\begin{split}
\rrr
=\ & z^\mmm g - (x^2 z^n -xy^2) \cg z^{\mmm-\ell}\\  %zhy8 add
=\ & z^\mmm\big(\cg x^2 z^{n-\ell} + x g_1(y,z) + g_0(x,y)\big)  %zhy8 add
- (x^2 z^n - xy^2) \cg z^{\mmm-\ell}\\  %zhy8 add
=\ & z^\mmm  ( x g_1(y,z)  +  g_0(y,z) ) %zhy8 add
+  \cg xy^2 z^{\mmm-\ell}. %zhy8 add
\end{split} %zhy8 add
\end{equation*} %zhy8 add
%zhy8 correct \begin{equation*}
%zhy8 correct \begin{split}
%zhy8 correct \rrr
%zhy8 correct =\ & z^\mmm g - (x^2 z^n -xy^2) \cg z^{\mmm-\ell}\\ %zhy7 z^{\mmm-1}\\
%zhy8 correct =\ & z^\mmm\big(\cg x^2 z^{n-\ell} + x g_1(y,z)\big)  %zhy7 z^{n-1}
%zhy8 correct - (x^2 z^n - xy^2) \cg z^{\mmm-\ell}\\ %zhy7 z^{\mmm-1}\\
%zhy8 correct =\ & xz^\mmm g_1(y,z) %zhy7 z^{\mmm-1}
%zhy8 correct +  \cg xy^2 z^{\mmm-\ell}.
%zhy8 correct \end{split}
%zhy8 correct \end{equation*}
As $\deg_x(\rrr) < 2 $ and $\deg_y(\rrr) < 3$, no term of
$\rrr$ is divisible by any element in $\LM{G}$, which means
$\rrr$ is the normal form of $z^\mmm g $ with respect to
$\GGG$. Recall that $z^\mmm g \in \III$, we have $\rrr = 0$,
which leads to $\cg=0$ and
$x g_1 + g_0 = 0$, %zhy8 was $g_1=0$
and so $g=0$. This is a
contradiction.
% but this is impossible since either $\deg_z(\rrr)=\mmm-1$ if $g_1=0$
% or $\deg_z(\rrr)\ge \mmm $ if $g_1\neq 0$.
Claim~1 is proved.
\hfill$\Box$
\newline

\item
{\itshape Claim~2.} For $n\in \ZZ_{+}$, the integral closure of the ideal
$\III = \ideal{x^3, y^3, x^2y,x^2z^n-xy^2 } \subset \kk[x,y,z]$
is $\barI %zhy5 was: $\barq
= \ideal{x^3, y^3, x^2 y, x^2 z^n,
xy^2}$.

{\itshape Proof of Claim~2.}
Let $\III_1 = \ideal{x^3, y^3} \subset \III$, then by
Proposition~\ref{prop:mo}, the monomial $xy^2$ is in the
integral closure of $\III_1$, and so
$xy^2$ is in the integral closure of $\III$.
Therefore $\III + \ideal{xy^2} \subset \barI$. %zhy5 was: \barq $.
On the other hand,
$\III + \ideal{xy^2}=\ideal{x^3, y^3, x^2 y, x^2 z^n, xy^2 }$ and it is integrally closed
according to Proposition~\ref{prop:mo}, which leads to
$\III + \ideal{xy^2}=\barI$. %zhy5 was: $\barq=J$.
Claim~2 is proved.
\hfill$\Box$
\newline

\item
{\itshape Claim~3.}
The ideal $\barI %zhy5 was: $\barq
= \ideal{x^3, y^3, x^2 y, x^2 z^n, xy^2 } \subset\kk[x,y,z]$ is not
primary, where $n\in \ZZ_{+}$.

{\itshape Proof of Claim~3.}
%zhy4 delete Let $\PPP = \ideal{x,y}$ be the radical of $\barq$.
%zhy7 The ideal $\barI$ %zhy5 was: $\barq$
%zhy7 is not primary,
Because
$x^2 z^n \in \barI$ %zhy5 was: $x^2 z \in \III$
while $x^2 \notin \barI $ and %zhy5 was: $x^2 \notin \III $ and
$(z^n)^{\mmm}\notin \barI$ %zhy5 was: $z^{\nnn}\notin \III$
for any $\mmm\in \ZZ_{+}$,
the ideal $\barI$ is not primary. Claim~3 is proved.
\hfill$\Box$
\end{itemize}

%zhy5 zijia moved the example from sec:moreexamples here.

In contrast, %zhy4 was: constrast,
there exist non-primary ideals whose integral closures are primary. For instance,
\[I=\langle x^2, y^2, xyz \rangle=\langle x^2,xy,y^2 \rangle   %zhy2 add
 \cap \langle x^2,y^2,z\rangle\] %Za: change to the primary intersection
 is not primary but its integral closure
 \[\barI = \langle x^2,y^2,xy \rangle\]
 is primary using Proposition~\ref{prop:mo}. %A  Macaulay2 computation result  is available in Appendix~\ref{app:m2}.

% \begin{center}
% {\ttfamily %zhy4 change font
% \renewcommand{\arraystretch}{1.2}
% \begin{tabular}{l@{\hspace{0.5em}}l}
% i1 : & R = QQ[x, y, z]; \\
% i2 : & I = ideal (x\^{}2, y\^{}2, x * y * z); \\
% i3 : & isPrimary (I) \\
% o3 = & false \\
% i4 : & associatedPrimes (I) \\
% o4 = & \{ideal($y,\ x$), ideal($z,\ y,\ x$)\} \\
% i5 : & isPrimary( integralClosure (I) ) \\ %zhy5 remove extra parentheses
% o5 = & true \\
% i6 : & associatedPrimes( integralClosure (I) ) \\
% o6 = & \{ideal($y,\ x$)\}
% \end{tabular}
% }%ttfamily
% \end{center}

%Notice that the number of embedded components can be increased if one can replace $z$ by other monomials or polynomials.  %{\todo{What does this integral mean? Add more explanations. }}

%zhy2 \subsection{More examples}\label{sec:moreexamples}
\subsection{Whitney stratification of the  hypersurface
defined by $ f=x^6 + y^6 + x^4 z t + z^3$}\label{sec:moreexamples}
\quad\quad
\vspace{2.4ex}

The following example was given in \cite{BriSpe1975} by Brian{\c{c}}on and Speder to show that Whitney equisingularity does not imply Zariski equisingularity. 
The Jacobian ideal $J$ of $f$  is primary and  its integral closure $\bar{J}$ of $J$ is not, which  gives another  counterexample to Krull's question.
Moreover, the embedded prime of $\bar{J}$ happens to be the vanishing ideal of the points where Whitney equisingularity fails.

% They  stated that the hypersurface defined by $f=0$ is Whitney regular along its singular locus at the origin but is not Zariski equisingular. We show  below that  the hypersurface defined by $f=0$ is not Whitney regular along its singular locus at other three points.

\begin{example}
\label{ex:original_whit}
Let $f = x^6 + y^6 + x^4 z t + z^3 \in \QQ[x,y,z,t] $, and its Jacobian ideal
\[J=\langle x^4 t + 3z^2,\ x^4z,\ y^5,\ 3x^5 + 2 x^3 z t \rangle.\]
 We first show that $J$ is a
primary ideal, while the integral closure
% \[\bar J=\langle ..  \rangle\]
%Za: add the polynomials of the integral closure of J
\begin{align*}
 \bar{J}  = \langle 3 x^2 y z+2 y z^2 t,3 x^3 z+2 x z^2 t,x^4 t+3 z^2,y^4 z,x^4 z,y^5,3 x^2 y^3+2 y^3 z t,3 x^3 y^2+2 x y^2 z t,\\
  9x^4 y-4 y z^2 t^2,3 x^5+2 x^3 z t,x^3 y z t,x^4 y^2,4 y^3 z t^3+27 y^3 z,x y^3 z t^2,2 x^3 y^2 t^2-9 x y^2 z,4 x y^4 t^3+27 x y^4  \rangle
\end{align*}
 is not primary. Its associated primes are
 \[\langle z,\ y,\ x \rangle ~\mbox{and}~ \langle z,\ y,\ x,\ 4t^3  + 27 \rangle.\]
  One can verify the result by Macaulay2, see Appendix~\ref{app:m3}.

Before we show that  the hypersurface defined by $f=0$ is not Whitney regular along its singular locus. Let us  recall the
definitions of Whitney's conditions~(a) and (b).

Let $\XXX$ and $\YYY$ be two smooth manifolds in $k^n$, $k=\RR$ or $\CC$. Suppose $\XXX \cap
\YYY = \emptyset$ and $\YYY$ is contained in the closure of $\XXX$.
Let $\boldp$ be a point in $\YYY$.
\begin{itemize}
\item
The  pair $(\XXX, \YYY)$ is said to satisfy
Whitney's condition~(a) at $\boldp$ if for any sequence
$\boldp_\epsilon\in \XXX$, $\boldp_\epsilon \rightarrow \boldp$
and $T_{\boldp_\epsilon}\XXX \rightarrow T$, then
$T_{\boldp}\YYY \subset T$.

\item
The  pair $(\XXX, \YYY)$ is said to satisfy
Whitney's condition~(b) at $\boldp$ if for any sequences
$\boldp_\epsilon \in \XXX$ and $\boldq_\epsilon \in \YYY$
such that
$\boldp_\epsilon\rightarrow \boldp$,
$\boldq_\epsilon\rightarrow \boldq$,
$T_{\boldp_\epsilon}\XXX \rightarrow T $,
and the lines $\overline{\boldp_\epsilon \boldq_\epsilon}$
converges to a line $\ell$ in the projective space $\PP^{n-1}$,
then $\ell \subset T$. Here $\overline{\boldp_\epsilon
\boldq_\epsilon}$ denotes the unique line through the two points
$\boldp_\epsilon$ and $ \boldq_\epsilon$.
\newline
\end{itemize}

Back to Example~\ref{ex:original_whit},
recall that $f = x^6 + y^6 + x^4 z t + z^3 \in \QQ[x,y,z,t]$ and the
Jacobian ideal of $f$ is
$J = \ideal{x^4 t + 3z^2,\ x^4z,\ y^5,\ 3x^5 + 2 x^3 z t}$.
The radical ideal of $J$ is $\sqrt{J}=\ideal{x,y,z}$.
Let $V_1 $ be the hypersurface  defined by  $f=0$ and $V_2$
be its singular locus. Then  we have
\begin{equation*}
V_2= \{(0,0,0,t)\mid t\in \CC \}\subset\CC^4.
\end{equation*}

We compute the set of points where
the pair $(V_1\setminus V_2,\ V_2)$  does
not satisfy Whitney's conditions
using the criterion in \cite[Lemma~2.8]{DJ2021}.
 We discover three such points defined by the variety
\begin{equation}\label{defv3}
V_3 =  \left\{(0,0,0,t) \in \CC^4\mid
4t^3 + 27 = 0 \right\}.
\end{equation}
These  points can also be found by running algorithms in  \cite{HelmerWhitStrat,HN2022}. 
The following is a formal proof for the statement.

\begin{thm}
The pair $(V_1\setminus V_2,\ V_2)$ does not  satisfy
Whitney's  condition~(a) and (b) at the points in $V_3$.
\end{thm}

\begin{proof} 
Since Whitney's  condition~(b) implies  Whitney's  condition~(a). In the following, we only prove the pair $(V_1\setminus V_2, V_2)$ does not  satisfy
Whitney's  condition~(a).

Let $\boldp = (0, 0, 0, \xi)$ be a point in $V_3$, where
$\xi = (-\sqrt[3]{27/4}) \omega $ and $\omega$ is one of the
cube roots of unity, i.e. $\omega^3=1$.
Consider the sequence of points
\begin{equation*}
\boldp_\epsilon =
(\epsilon, 0, \ccc \, \epsilon^2, \xi)
\text{ where } \epsilon \neq 0
\text{ and } \ccc = (\sqrt[3]{1/2})\omega^2. 
\end{equation*}
Note that $\xi = -3\ccc^2$ and $\ccc^3 = 1/2$. 
For any $\epsilon\neq 0$, we have
$\boldp_\epsilon \in V_1\setminus V_2$ because
\begin{equation*}
\begin{split}
& \epsilon^6 + 0 + \epsilon^4 (\ccc\epsilon^2)\xi + (\ccc\epsilon^2)^3
\\
=\ & (1 + \ccc\xi + \ccc^3) \epsilon^6 \\
=\ & (1 - 3\ccc^3 + \ccc^3) \epsilon^6 \\
=\ & (1-2\ccc^3)=0.
\end{split}
\end{equation*}
Also, $\boldp_\epsilon \rightarrow \boldp$ 
as $\epsilon \rightarrow 0$.
The Jacobian matrix of $f$ is
\begin{equation*}
\begin{split}
\jac{f} &= %\left[ 6x^5 + 4x^3 zt, \ 6y^5,\ 3z^2+x^4 zt,\ x^4 z \right]^T
\begin{pmatrix}
\partial f / \partial x, & %\frac{\partial f} {\partial x} \\
\partial f / \partial y, & %\frac{\partial f} {\partial y} \\
\partial f / \partial z, & %\frac{\partial f} {\partial y} \\
\partial f / \partial t  %\frac{\partial f} {\partial y} \\
\end{pmatrix} \\
&=
\begin{pmatrix}
6x^5 + 4x^3 zt, &
6y^5, &
3z^2+x^4 t, &
x^4 z
\end{pmatrix}
\end{split}
\end{equation*}
Evaluating  the  Jacobian
matrix of $f$ at $\boldp_\epsilon=(\epsilon, 0, \ccc  \epsilon^2,\xi), ~ \xi=-3\ccc^2$, we get
\begin{equation*}
\begin{split}
\jac{f}\rvert_{\boldp_\epsilon}
&=
\begin{pmatrix}
6 \epsilon^5 + 4\epsilon^3(\ccc\epsilon^2)\xi , &  0, &
3 (\ccc\epsilon^2)^2 + \epsilon^4\xi,
& \epsilon^4(\ccc\epsilon^2)
\end{pmatrix} \\
&=
\begin{pmatrix}%
6 \epsilon^5 -12 \ccc^3  \epsilon^5 , &  0, &
3 \ccc^2 \epsilon^4+  \epsilon^4 (-3 \ccc^2),
& \ccc\epsilon^6
\end{pmatrix} \\
&=
\begin{pmatrix}
0, & 0, & 0, &  \ccc\epsilon^6
\end{pmatrix}.
\end{split}
\end{equation*}
The tangent space of $V_1$ at $\boldp_\epsilon$, denoted as
$T_{\boldp_\epsilon} V_1$,
is the kernel of $\jac{f}\rvert_{\boldp_\epsilon} $.
Therefore, we have
\[ T_{\boldp_\epsilon} V_1 = \{(x,y,z,t)\in \CC^4 \mid t = 0\},\]
 which
is not related to $\epsilon$.
Let
\[T = \{(x,y,z,t)\in \CC^4 \mid t=0 \},\]
then
\[T_{\boldp_\epsilon}V_1 \rightarrow T  ~\text{as}~
\epsilon\rightarrow 0.\]
Recall that
\[V_2 = \{(0,0,0,t)\in \CC^4\}\]
a linear space, so the tangent space of $V_2$ at $\boldp$ is $V_2$
itself.  Now we have
\[T_{\boldp} V_2 = V_2 \not\subset
T,\]
 which violates Whitney's condition~(a).
\end{proof} %zhy5

Furthermore, it is interesting to know that Example~\ref{ex:original_whit}
also provides an example to show that
Whitney stratification is different from the
stratification given by isosingular sets
in \cite{HW2013}.  Let $V_0 = \{(0,0,0,0) \}$ be the origin.
The stratification of $V_1$
given by isosingular sets is
\[\Sigma_S := (V_1\setminus V_2,\ V_2 \setminus V_0,\ V_0),\]
and the minimal Whitney Stratification of $V_1$ is
\[\Sigma_W :=(V_1\setminus V_2,\ V_2 \setminus V_3,\ V_3).\]
The origin  is an isosingular point \cite[Definition 5.1]{HW2013} while all three points in $V_3$ are not.  Moreover, 
the hypersurface $V_1$ fails Zariski's equisingular conditions at the origin  \cite{speder1975equisingularite}.
 Since Zariski equisingularity implies Whitney equisingularity \cite{ParPau2017,speder1975equisingularite}, we know that 
  the hypersurface $V_1$  fails to be Zariski equisingular along $V_2$ at least at  four points including the origin and three points in  $V_3$~(\ref{defv3}).

 %For the stratification $\Sigma_S$, the point $\boldp = (0, 0, 0, \xi)$ is contained in the stratum $V_2\setminus V_0$ of dimension one,    while for the minimal Whitney stratification of $V_1$, the point $\boldp$ is contained in a stratum $V_3 =  \left\{(0,0,0,t) \in \CC^4\mid  4t^3 + 27 = 0 \right\}$ of dimension zero.  Recall that as Zariski showed in \cite{zariski1977elusive}, the hypersurface $V_1$  fails to be equisingular along its singular locus $V_2$ at $t=0$.
%  Moreover, Zariski equisingularity implies Whitney equisingularity \cite{ParPau2017,speder1975equisingularite}.
%  Therefore, the hypersurface $V_1$  fails to be Zariski equisingular along $V_2$ at least at  four points  of $V_3\cup \{(0,0,0,0)\}$.
\end{example}

In Example~\ref{ex:original_whit}, the hypersurface defined by
$f = x^6 + y^6 + x^4 z t + z^3 $ is
in $4$-dimensional space. Let $H_f$ be the hypersurface defined by
$f$ in $\RR^4$. We cannot
draw %zhy5 was plot
a figure of $H_f$,
however, we can consider the following “slice'' of the
$H_f$:
\begin{equation*}
H_f \cap \{(x,y,z,t)\in \RR^4 \mid y=0\},
\end{equation*}
which is homeomorphic to the surface
\[H_g = \{(x,z,t)\in \RR^3 \mid g = x^6 + x^4 z t + z^3 = 0\}\]
in $3$-dimensional space. 
We can see how the surface $H_g$ looks like around the point
$(0,0,-\sqrt[3]{27/4})$ (see
Figure~\ref{fig:mouse}), which
gives a glimpse of how the hypersurface $H_f$
behaves around the real point $(0,0,0,-\sqrt[3]{27/4})$. We also find that the Jacobian ideal of the polynomial $g = x^6 + x^4 z t + z^3$ is a counterexample to Krull's question. 
\begin{figure}[tbhp]
  \centering
  \includegraphics[width=0.32\textwidth]{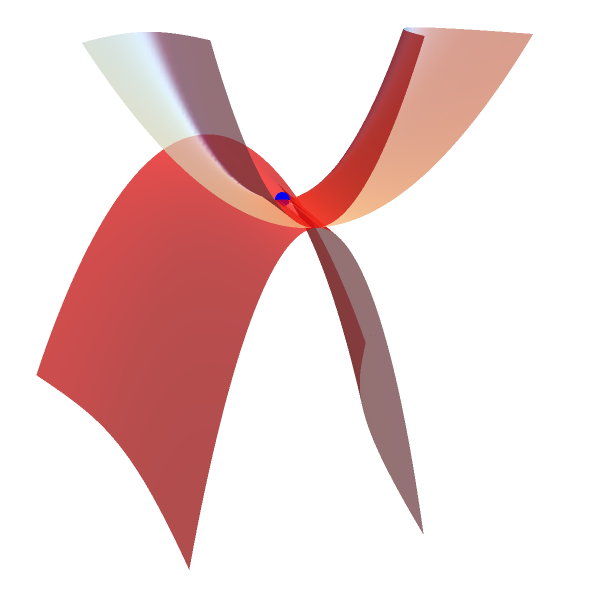} %zhy6 change figure
  %zhy1 was \includegraphics[width=10cm]{figure/twsurface.jpg}
  \caption{The surface defined by $g=x^6+x^4zt+z^3$ and the blue
point $(0,0,-\sqrt[3]{27/4})$.}
  \label{fig:mouse} %zhy2 was:\label{fig1}
\end{figure}

Moving the hyperplane
$\HHH_{t_0} = \{(x,z,t) \in \RR^3 \mid t = t_0 \}$ along
the $t$-axis,
the intersection  $\HHH_{t_0} \cap H_g$ changes as
(see Figure~\ref{fig:y0_ywhitney_yminus2}): %zhy6 add
%zhy5 delete: (see Figure~\ref{fig:y0_ywhitney_yminus2}):
\begin{equation*}
\HHH_{t_0}\cap H_g =
\left\{
\begin{array}{@{}ll}
\text{ one parabola,}    &  \text{if  } t_0 > -\sqrt[3]{27/4}, \\
\text{ two parabolas,}   &  \text{if  } t_0 =  -\sqrt[3]{27/4}, \\
\text{ three parabolas,} &  \text{if  } t_0 < -\sqrt[3]{27/4},
\end{array}
\right.
\end{equation*}
which also explains why the point %zhy5 change
$(0,0,-\sqrt[3]{27/4})$ is different from %zhy5 change
other points on the $t$-axis. %zhy5 change
%zhy6  change figures
\begin{figure}[tbhp]
   \centering
   \subfigure[]{\includegraphics[width=0.32\textwidth]{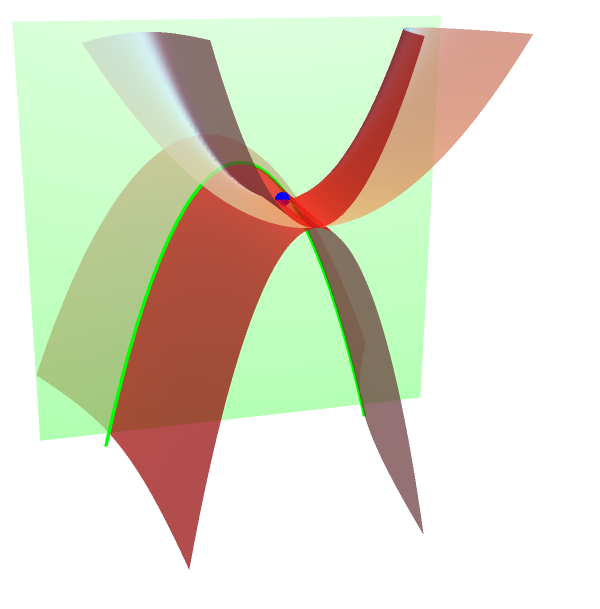}}
   \subfigure[]{\includegraphics[width=0.32\textwidth]{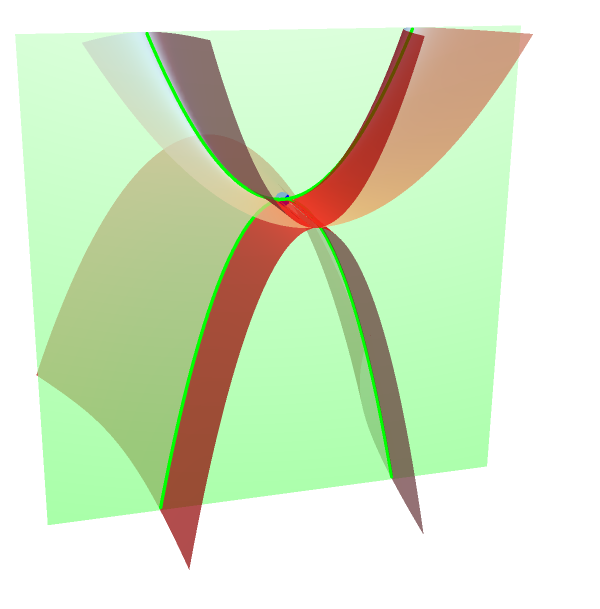}}
   \subfigure[]{\includegraphics[width=0.32\textwidth]{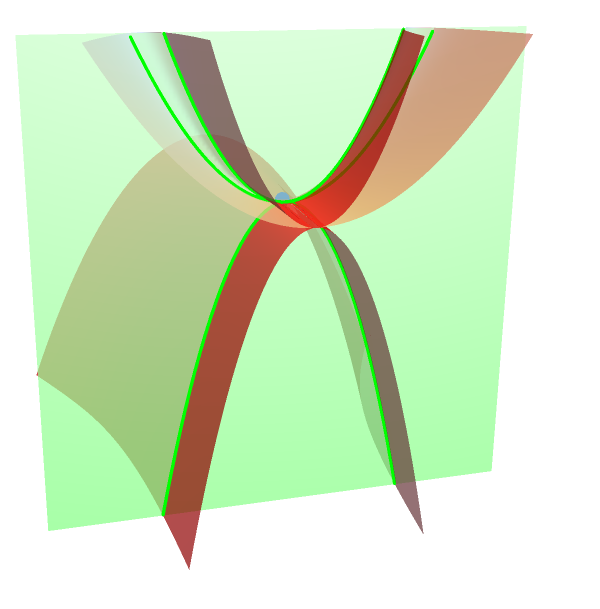}}
\caption{(a) The intersection $H_g\cap \{t=0\}$ consists of one parabola. (b) The
intersection $H_g\cap \{t=-\sqrt[3]{27/4}\}$ consists of two parabolas.
(c) The intersection $H_g\cap \{t=-2 \}$ consists of three parabolas. }
%zhy9 change \caption{The intersection of the green hyperplane and the red surface
%zhy9 change changes significantly when the hyperplane cross the blue point
%zhy9 change $(0,0,-\sqrt[3]{27/4})$.}
 \label{fig:y0_ywhitney_yminus2}
 \end{figure}
%zhy5 was: which also explains why the smooth part of the surface and the
%zhy5 was: singular locus of the surface does not satisfy Whitney's condition (a)
%zhy5 was: at the point $(0,0,-\sqrt[3]{27/4})$.
%zhy5 delete \newline
%zhy5 delete \begin{figure}[tbhp]
%zhy5 delete   \centering
%zhy5 delete   \subfigure[]{\includegraphics[width=0.3\textwidth]{Figures/y_equal_0.eps}}
%zhy5 delete   \subfigure[]{\includegraphics[width=0.3\textwidth]{Figures/y_whit.eps}}
%zhy5 delete   \subfigure[]{\includegraphics[width=0.3\textwidth]{Figures/y_equal_minus2.eps}}
%zhy5 delete %zhy5 was: \caption{(a) $V_1 \cap \{y=0\}$.  (b) $V_1 \cap
%zhy5 delete %zhy5 was: \{y=-\sqrt[3]{27/4}\} $.  (c) $V_1 \cap \{y=-2\}$.}
%zhy5 delete \caption{(a) $H_g\cap \{t=0\}$. (b) $H_g\cap \{t=-\sqrt[3]{27/4}\}$.
%zhy5 delete (c) $H_g\cap \{t=-2 \}$.} %zhy5 change $V_1$ to $H_g$ and y to t.
%zhy5 delete \label{fig:y0_ywhitney_yminus2}
%zhy5 delete \end{figure}

%\input{ex_mouse}

%zhy2 %zhy1 add
%zhy2 \input{extrip}

\section{Conclusion}

  We give counterexamples of primary ideals whose integral closures are not primary. Some of our examples are
valid for any characteristics; this fills the gap left  by Huneke's example that works only in positive
characteristics.  Another counterexample is the Jacobian ideal $J$ of $f(x,y,z,t) =  x^6 + y^6 + x^4 z t + z^3$. The hypersurface $V_1$ defined by $f=0$ is given by Brian{\c{c}}on and Speder to show that  Whitney equisingularity  does not imply  Zariski equisingularity for the set germs at the origin.
 Let $V_2= \{(0,0,0,t)\}$ be the singular locus of $V_1$, $\bar{J}$ be the integral closure of $J$, and $V_3= \{(0,0,0,t)\mid 4t^3+27=0 \}$ be the variety of the embedded prime of $\bar{J}$.
 It turns out that the variety $V_3$ contains exactly all the points where the pair $(V_1\setminus V_2, V_2)$ fails Whitney's conditions.  %zhy12 was: Although  the pair $(V_1 \setminus V_2, V_2)$ satisfies Whtiney's conditions at the origin, we show that the pair fails Whitney's conditions  at $V_3$.
Moreover, we show that the stratification defined by  isosingular sets is different from Whitney stratification. In particular, the pair $(V_1 \setminus V_2, V_2)$  satisfies  Whitney's conditions at the origin $V_0 = \{(0,0,0,0) \}$, while $V_0$ is  an isosingular point. On the other hand,   all three points in $V_3$  are not isosingular points.  It is interesting to notice that the hypersurface is not Zariski equisingular along its singular locus at both $V_0$ and $V_3$.

 %In other words, the Whitney stratification differs from the stratification given by isosingular sets. }% Since the hypersurface is not Zariski equisingular along its singular locus at $V_0$ and $V_3$, it is naive to conjecture that Zariski equisingularity implies the simultaneously Whitney equisingularity and isosingular sets. }
 %The integral closure of ideals is also
% useful for the well-known Eisenbud-Mazur conjecture.
% Huneke asks, ``Is $f$ contained in the maximal ideal times the
% integral closure of ideal generated by  all
% $\frac{df}{dx_i}$.'' Huneke has shown that if the answer is
% always yes, then the Eisenbud-Mazur conjecture on evolutions
% is true.}

%
%\input{2basic}
%
%\input{3examples}
%
%\input{4conclusion}

\section*{Acknowledgements}
 We would like to express our great thanks to Prof. Huneke and Prof.
Swanson for the discussion on the latest state-of-the-art for the
counterexample in primary ideals regarding the question from Krull.
We also would like to thank Prof. Adam Parusiński
 for pointing out   that Briançon and Speder's 
example concerns
equisingularity of set germs at the origin.
 This work was started during the summer seminar
\emph{Singularities of Differentiable} \emph{Maps} at Shenzhen University in July
2021. Lihong  Zhi is  supported by the National Key Research Project of China 2018YFA0306702 and the National Natural Science Foundation of China 12071467.
% \newline %zhy2 add

%zhy10 was: \bibliographystyle{elsarticle-harv}%zhy5 was: \bibliographystyle{elsarticle-num}
\bibliographystyle{abbrv}
\bibliography{ws,zhyang} %add my bib file

%zhy4 add appendix

\appendix 

\section{}
 
 \subsection{}\label{app:m1}
\begin{center}
{\ttfamily 
\renewcommand{\arraystretch}{1.2}
\begin{tabular}{l@{\hspace{0.5em}}l}
i1 : & R = QQ[x, y, z]; \\
i2 : & I = ideal (x\^{}3,  y\^{}3,  x\^{}2 * y,  x\^{}2 * z - x * y\^{}2); \\
i3 : & isPrimary (I) \\
o3 = & true \\
i4 : & primaryDecomposition (integralClosure (I)) \\
o4 = & \{ideal($x^2,\ y^3,\ xy^2$), ideal($z,\ y,\ x^3$)\} \\
\end{tabular}
}
\end{center}
Here, ``\texttt{QQ}'' is the rational number field $\QQ$,
and %same notation will be used in the remaining of this paper. N
  one can replace  $\QQ$
by  other computational fields of positive characteristics available in Macaulay2.

% \subsection{}\label{app:m2}
% \begin{center}
% {\ttfamily %zhy4 change font
% \renewcommand{\arraystretch}{1.2}
% \begin{tabular}{l@{\hspace{0.5em}}l}
% i1 : & R = QQ[x, y, z]; \\
% i2 : & I = ideal (x\^{}2, y\^{}2, x * y * z); \\
% % i3 : & isPrimary (I) \\
% % o3 = & false \\
% i3 : & primaryDecomposition (I) \\
% o3 = & \{ideal($y^2,\ xy,\ x^2$), ideal($z,\ y^2,\ x^2$)\} \\
% i4 : & integralClosure (I)  \\ %zhy5 remove extra parentheses
% o4 = & \{ideal($y^2,\ xy,\ x^2$)\}
% \end{tabular}
% }%ttfamily %Za: remove redundant computations
% \end{center}

\subsection{}\label{app:m3}
\begin{center}
{\ttfamily
\renewcommand{\arraystretch}{1.2}
\begin{tabular}{l@{\hspace{0.5em}}l}
i1 : & R = QQ[x, y, z, t]; \\
i2 : & I = ideal (x\^{}6 + y\^{}6 + x\^{}4 * z * t + z\^{}3); \\
i3 : & J =  trim ideal ( jacobian (I) ) \\ %zhy10 was: (singularLocus I) \\
o3 = & ideal($x^4 t + 3z^2,\ x^4z,\ y^5,\ 3x^5 + 2 x^3 z t $) \\
% i4 : & isPrimary (J) \\
% o4 = & true \\
i4 : & associatedPrimes (J) \\
o4 = & \{ideal ($z,\ y,\ x$)\} \\
% i6 : & isPrimary ( integralClosure (J) ) \\
% o6 = & false \\
i5 : & associatedPrimes ( integralClosure (J) ) \\
o5 = & \{ideal ($z,\ y,\ x$), ideal ($z,\ y,\ x,\ 4t^3  + 27$)\}
\end{tabular}
}%ttfamily
\end{center}

\section{}
\label{sec:append}

\begin{defn}\cite[Definition~2.3]{shimoyama96}
\label{def:pseudo}
An ideal $I$ of a polynomial ring over $\QQ$ is called a
pseudo-primary ideal if $\sqrt{I}$ is a prime ideal.
\end{defn}

\begin{prop}
{\upshape \cite[Proposition~2.11]{shimoyama96}}
\label{prop:pseudo_primary}
Let $I$ be a pseudo-primary ideal in $\QQ[x_1, \ldots, x_n]$ with
$\sqrt{I}=P$ and let $Q$ be the primary component of $I$ with
$\sqrt{Q}=P$. Suppose that a subset $U$ of $\{x_1, \ldots, x_n \}$ is
a maximally independent set modulo $P$. Then
$Q = I \QQ(U)[\{x_1,\ldots,x_n\}\setminus U] $.
\end{prop}

\begin{prop}
{\upshape \cite[Proposition~8.94]{BeckerWei1993}} %zhy5 delete Lemma~A.8]{shimoyama96},
\label{prop:beckerwei} %zhy5 was \label{lem:a8}
%zhy5 change the field $QQ$ to $k$
Let $\kk$ be a field, $I$ be an ideal of
$\kk[x_1,\ldots,x_n]$, %zhy5 was $R:= \QQ[x_1,\ldots,x_n]$,
$U$ be any subset of $\{x_1,\ldots,x_n\}$, and $G$ be a \groebner basis of
$I$ with respect to
an inverse block order $<$ in $\{x_1,\ldots,x_n\}$ such that $U \ll
\{x_1,\ldots,x_n\}\setminus U$.
Set
\begin{equation*}
f=\lcm\{\lcu{g}\mid g\in G\},
\end{equation*}
where $\lcu{g}\in \kk[U] $ is the leading coefficient of $g$ as an
element in $k(U)[\{x_1,\ldots,x_n \}\setminus U ]$ with respect to
the restriction $<^{\prime}$ of $<$ to
$\{x_1,\ldots,x_n\}\setminus U$. Then
$I k(U)[\{x_1,\ldots,x_n\}\setminus U]= I : f^\infty  $.
%zhy5 rewrite: For each $g\in R$, let $\lcu{g}$ be
%zhy5 rewrite: the leading coefficient %zhy5 was: head coefficient
%zhy5 rewrite: of $g$ in $\QQ[U]$ as an element in
%zhy5 rewrite: $\QQ(U)[\{x_1,\ldots,x_n \}\setminus U ]$ with respect to the order
%zhy5 rewrite: $<^{\prime}$, where $<^{\prime}$ is the restriction of $<$ to
%zhy5 rewrite: $\{x_1,\ldots,x_n\}\setminus U$. Then $I\QQ(U)[\{x_1,\ldots,x_n
%zhy5 rewrite: \}\setminus U]= I : f^\infty  $, %zhy5 was: I R_f\cap R$,
%zhy5 rewrite: where $f=\lcm\{\lcu{g}\mid g\in G\}$.
\end{prop} %zhy5 was: \end{lem}

%zhy5 delete \begin{lem}
%zhy5 delete {\upshape \cite[Exercise 1.4.2]{GP2008singular}}
%zhy5 delete \label{lem:ex}
%zhy5 delete Let $R$ be a ring, $I\subset R$ an ideal and $f\in R$.  Then
%zhy5 delete \begin{equation*}
%zhy5 delete I R_f \cap R = I : \ideal{f} = \{g\in R \mid gf^n \in I
%zhy5 delete \text{ for some n}\}.
%zhy5 delete \end{equation*}
%zhy5 delete \end{lem}

\begin{lem}
\label{lem:saturation} %zhy5 change the rational number field $\QQ$ toa general field $k$.
Let $R=\kk[x,y,z]$, $n\in \ZZ_{+}$, and $\III = \ideal{x^3, y^3,
x^2y,x^2z^n-xy^2
}\subset R$ be an ideal, $\PPP=\sqrt{\III} = \ideal{x,y}$. Let $\qqq$
be the primary component of $\III$ with $\sqrt{\qqq}=\PPP$. Then
\begin{equation*}
\qqq=\III : z^\infty %zhy5 was \qqq=\III :\ideal{z}
\end{equation*}
where $\III : z^\infty %zhy5 was where $\III : \ideal{z}
= \{g\in \kk[x,y,z] \mid z^\mmm g \in \III
\text{ for some } \mmm \}$.
\end{lem}
\begin{proof}
Definition~\ref{def:pseudo} and  %zhy5 add
the proof of Proposition~\ref{prop:pseudo_primary} %zhy5 add
are valid for polynomial rings over any field. %zhy5 add
$\III$ is a pseudo-primary ideal since
$\PPP =\sqrt{\III} = \ideal{x,y}$ is a prime ideal. %zhy5 delete by Definition~\ref{def:pseudo}.
The set $U=\{z\}$ is a maximally independent set modulo $\PPP$.
By Proposition~\ref{prop:pseudo_primary}, $\qqq =
\III k(z)[x,y]$. %zhy5 was \III\QQ(u)[x,y]$.

The set of generators $\GGG = \{x^3, y^3, x^2y,x^2z^n-xy^2 \}$ is a
\groebner basis of $\III$ with respect to the lexicographic order
$x>y>z$.  The leading coefficients %zhy5 was: head coefficients
of the polynomials in  $ \GGG $ as
elements in $k(z)[x,y] $ are %zhy5 was $\QQ(z)[x,y] $ are
\begin{equation*}
\lcu{x^3}= \lcu{y^3} = \lcu{x^2y} = 1,\  \lcu{x^2z^n-xy^2} = z^n.
\end{equation*}
Therefore $z^n = \lcm\{\lcu{g} \mid g \in \GGG \}$.
By Proposition~\ref{prop:beckerwei},
$ \qqq = \III : (z^n)^\infty =  \III : z^\infty $. %zhy7 was: $ \qqq = \III : z^\infty $.
%zhy5 was By Lemma~\ref{lem:a8},
%zhy5 was $\qqq = \III R_z\cap R$ where $R_z = \{g/z^\mmm
%zhy5 was \mid g\in R \}$, then conclusion follows from Lemma~\ref{lem:ex}.
\end{proof}

% \end{appendices}

% \bibliographystyle{plain}

\hspace{0.5cm}

% College of Mathematics and Statistics, Shenzhen
% University, Shenzhen, Guangdong 518060, China
%
% \emph{E-mail address}: \url{nan.li@szu.edu.cn}
%
% KLMM, Academy of Mathematics and Systems Science, Chinese
% Academy of Sciences, Beijing 100190, China
%
% \emph{E-mail address}: \url{lizijia@amss.ac.cn}
%
% College of Mathematics and Statistics, Shenzhen
% University, Shenzhen, Guangdong 518060, China
%
% \emph{E-mail address}: \url{zhyang@szu.edu.cn}
%
%
% KLMM, Academy of Mathematics and Systems Science, Chinese
% Academy of Sciences, Beijing 100190, China, University of Chinese Academy of Sciences, Beijing 100049, China
%
% \emph{E-mail address}: \url{lzhi@mmrc.iss.ac.cn}

College of Mathematics and Statistics, Shenzhen
University, Shenzhen, Guangdong 518060, China \quad \emph{E-mail address}: \url{nan.li@szu.edu.cn};  \url{zhyang@szu.edu.cn}

KLMM, Academy of Mathematics and Systems Science, Chinese
Academy of Sciences, Beijing 100190, China
 \quad
\emph{E-mail address}: \url{lizijia@amss.ac.cn}

% College of Mathematics and Statistics, Shenzhen
% University, Shenzhen, Guangdong 518060, China
%  \quad
% \emph{E-mail address}: \url{zhyang@szu.edu.cn}

KLMM, Academy of Mathematics and Systems Science, Chinese
Academy of Sciences, Beijing 100190, China, University of Chinese Academy of Sciences, Beijing 100049, China

\emph{E-mail address}: \url{lzhi@mmrc.iss.ac.cn}

\end{document}